\documentclass{article}
\usepackage[utf8]{inputenc}
\usepackage{amssymb}
\usepackage{amsmath}
\usepackage{amsfonts}
\usepackage{amsthm}
\usepackage{xcolor}
\usepackage{amsxtra}
\usepackage{amscd}
\usepackage{cite}
\usepackage{eucal}
\usepackage[all]{xy}
\usepackage{graphicx}
\usepackage{pifont}
\usepackage{comment}
\usepackage{tikz-cd}
\usepackage{enumitem}
\usepackage{euler}
\usepackage{latexsym}
\usepackage{framed}
\usepackage{fullpage}
\usepackage{hyperref}
\usepackage{kbordermatrix}

\makeatletter
\renewenvironment{proof}[1][\proofname] {\par\pushQED{\qed}\normalfont\topsep6\p@\@plus6\p@\relax\trivlist\item[\hskip\labelsep\bfseries#1\@addpunct{.}]\ignorespaces}{\popQED\endtrivlist\@endpefalse}
\makeatother
\theoremstyle{plain}
\newtheorem{theorem}{Theorem}[subsection]
\newtheorem{proposition}[theorem]{Proposition}
\newtheorem{lemma}[theorem]{Lemma}
\newtheorem{cor}[theorem]{Corollary}
\theoremstyle{definition}
\newtheorem{definition}[theorem]{Definition}
\newtheorem{ex}[theorem]{Example}
\newtheorem{remark}[theorem]{Remark}
\newcommand{\nc}{\newcommand}
\newcommand{\Z}{\mathbb{Z}}

\newcommand{\C}{\mathbb{C}}

\newcommand{\ba}{\overline}

\nc{\on}{\operatorname}
\nc{\Hom}{\on{Hom}}
\nc{\Ind}{\on{Ind}}
\nc{\Irr}{\on{Irr}}
\nc{\Res}{\on{Res}}

\title{Characters of Renner Monoids and their Hecke Algebras}
\author{Andrew Hardt \and Jared Marx-Kuo \and Vaughan McDonald \and John M. O'Brien \and Alexander Vetter}

\newcommand{\Addresses}{{
  \bigskip
  
  \noindent Andrew Hardt, \textsc{Department of Mathematics, University of Minnesota, Minneapolis, MN 55455}\par\nopagebreak
  \textit{E-mail address}: \texttt{hardt040@umn.edu}

  \medskip

  \noindent Jared Marx-Kuo, \textsc{Department of Mathematics, Stanford University, Stanford, CA 94305}\par\nopagebreak
  \textit{E-mail address}: \texttt{jmarxkuo@stanford.edu}

  \medskip

  \noindent Vaughan McDonald, \textsc{Department of Mathematics, Harvard University, Cambridge, MA 02138}\par\nopagebreak
  \textit{E-mail address}: \texttt{vmcdonald@college.harvard.edu}

  \medskip

  \noindent John M. O'Brien, \textsc{Department of Mathematics, University of Minnesota, Minneapolis, MN 55455}\par\nopagebreak
  \textit{E-mail address}: \texttt{obri0741@umn.edu}
  
  \medskip
  
  \noindent Alexander Vetter, \textsc{Department of Mathematics, University of Pennsylvania, Philadelphia, PA 19104}\par\nopagebreak
  \textit{E-mail address}: \texttt{avett@sas.upenn.edu}
  
}}

\date{\today}

\begin{document}

\maketitle

\begin{abstract}
This paper gives a general algorithm for computing the character table of any Renner monoid Hecke algebra, by adapting and generalizing techniques of Solomon used to study the rook monoid. The character table of the Hecke algebra of the rook monoid (i.e., the Cartan type $A$ Renner monoid) was computed earlier by Dieng, Halverson, and Poladian using different methods. Our approach uses analogues of so-called A- and B-matrices of Solomon. In addition to the algorithm, we give explicit combinatorial formulas for the A- and B-matrices in Cartan type $C$ and use them to obtain an explicit description of the character table for the type $C$ Renner monoid Hecke algebra.
\end{abstract}

\medskip


\medskip


\section{Introduction}

\subsection{Background}
This paper concerns the character tables of Renner monoids and Hecke algebras associated to finite reductive algebraic monoids. Reductive algebraic groups have long been a centerpiece of mathematics, and the related study of reductive algebraic monoids has become a fully-fledged subject of its own. The subject was created and developed in a series of papers by Renner, Putcha, and their collaborators (see \cite{renner} and its references), and has since been studied by Solomon \cite{solomonheckealgebra, solomon1995introduction}, Halverson \cite{halverson2004representations}, Okni{\'n}ski \cite{okninski}, and others. The Braverman-Kazhdan program \cite{bravermankazhdan, ngo} has sparked renewed interest in reductive monoids: they conjecturally admit a theory of harmonic analysis suitable for integral representations of Langlands L-functions.

An algebraic monoid over an algebraically closed field $K$ is a Zariski closed submonoid of Mat$_n(K)$, the monoid of $n\times n$ matrices with entries in $K$. A reductive monoid is an algebraic monoid whose group of units is an algebraic group; equivalently, it is an algebraic monoid that is regular as a semigroup. The structure of reductive monoids mirrors that of reductive groups in many ways, and a good portion of the research in the former involves connections and analogues to the latter.

In particular, the Renner monoid of a reductive monoid plays the same role as the Weyl group of a reductive group, and it contains the Weyl group as its group of units. One can similarly define the Hecke algebra of a Renner monoid, and by the Putcha-Tits' Deformation Theorem \cite[Theorem~4.1]{putchaheckealgebra}, this is abstractly isomorphic to the monoid algebra of the Renner monoid itself.

A reductive monoid over a finite field is defined to be the set of Frobenius fixed points of a reductive monoid over the algebraic closure. Like reductive groups, much of the reductive monoid structure still holds in the finite case: the Renner monoid, Weyl group, Borel subgroup, cross-sectional lattice, and Bruhat-Renner decomposition all remain intact after passing to the Frobenius fixed points. The finite Hecke algebra is again a (trivial) deformation of the group algebra of the Renner monoid, and in a sense the Hecke algebra is a ``link'' between the reductive group and its Renner monoid. Representations of the Renner monoid correspond to those of its Hecke algebra, and a monoid version of the Borel-Matsumoto theorem \cite{monoidborelmatsumoto} tells us that these in turn correspond to certain representations of the associated reductive monoid.

In this paper, we study the character theory of finite Renner monoids and their Hecke algebras. Much work has been done in Cartan type $A$. Here, the Renner monoid is the rook monoid $R_n$, the set of all partial permutation matrices ($\{0,1\}$-matrices with at most one 1 in each row and column) of size $n$ \cite{solomon1995introduction}. Munn \cite{munn1957characters} computed the character table of the rook monoid. Solomon \cite{solomon2002representations} then reformulated this in terms of A- and B-matrices, which encode the Renner monoid character table in terms of character tables of certain subsemigroups. Steinberg \cite{steinberg} and Li, Li, and Cao \cite{liRennerMonoid} computed the character table for any Renner monoid. Dieng, Halverson, and Poladian \cite{Dieng2003} computed the character table of the rook monoid Hecke algebra, computing Frobenius and Murnaghan-Nakayama rules. Outside of Cartan type $A$, few results on the theory of Renner monoid Hecke algebras are known. In this paper, we prove a general algorithm for computing the character table of any generic Renner monoid Hecke algebra, originally defined by Godelle \cite{godelle}. In the case of the type $C_n$ Renner monoid, the symplectic Renner monoid $RSp_{2n}$, we turn this into specific formulas. We describe $RSp_{2n}$ further in Section \ref{rspdefinitionsection}. To explain our results more precisely, we recall Solomon's approach in type $A$.

\subsection{Solomon's A- and B-Matrices}

Let $R$ be a Renner monoid, $W$ be the Weyl group that is its group of units, and $\Lambda$ its cross-sectional lattice of idempotents. Then $R = \langle W,\Lambda\rangle$, and the representation theory of $R$ is governed by the representation theory of certain subsemigroups $W_e$ that are groups with identity $e\in\Lambda$. In particular, the irreducible representations $\rho$ of $R$ correspond bijectively to (and restrict to) the representations $\ba{\rho}$ of $W_e$, $e\in\Lambda$.

Let $M$ be the character table of $R$, and let $Y$ be the block diagonal matrix \[Y := \begin{bmatrix} H_n & & & \dots & 0  \\ & H_{n-1}&& & \dots \\ && \ldots \\ \dots &&& H_1 \\  0&\dots &&& H_0\end{bmatrix},\] where the $H_i$ are the character tables of the $W_e, e\in\Lambda$. $M$ and $Y$ are both invertible, so we can write $M = AY = YB$ for invertible matrices $A$ and $B$. Solomon first introduced these matrices in \cite{solomon2002representations} in type $A$ and showed that they satisfy natural combinatorial formulas \cite[Theorems~3.5,~3.11]{solomon2002representations}. Solomon's use of these matrices as a bridge between a Renner monoid and certain important subsemigroups is instrumental to our results.

\subsection{Main Results of the Paper}

Our results are divided into three parts. First, we prove a simple algorithm for computing the character table of a Renner monoid Hecke algebra, using the $q$-analogue of the Solomon decomposition: $M_q = A_qY_q = Y_qB_q$. Similarly to the group case, we let \[Y_q := \begin{bmatrix} H_n^* & & & \dots & 0  \\ & H_{n-1}^*&& & \dots \\ && \ldots \\ \dots &&& H_1^* \\  0&\dots &&& H_0^*\end{bmatrix},\] where the $H_i^*$ are the character tables of the group Hecke algebras $\mathcal{H}(W_e), e\in\Lambda$. The B-matrix of a Renner monoid encodes certain restriction multiplicities, and as a result we show that $B$ and $B_q$ are the same (Proposition \ref{heckealgebramultiplicities}). Thus, the character table $M_q$ can be computed from the knowledge of $M$ and $Y_q$, both of which are known for all types.

Second, we echo Solomon and give combinatorial descriptions of the A- and B-matrices for the Renner monoid $RSp_{2n}$. Both matrices are block upper triangular, with identity blocks on the diagonal, and can be described using partition combinatorics. The A-matrix entries are products of binomial coefficients, while the B-matrix entries are given as a sum of Littlewood-Richardson coefficients arising from a Pieri rule.

Finally, we use Solomon's method to give combinatorial formulas for the Hecke algebra character table of $RSp_{2n}$. Starkey \cite[Theorem~9.2.11]{geck2000characters} gave a combinatorial rule for the type $A$ group Hecke algebra character table, which was generalized by Shoji \cite{shoji} to Ariki-Koike algebras. We show that these formulas result in a similar formula for the Hecke algebra character table of $RSp_{2n}$.

\subsection{Outline}

In Section \ref{rennermonoidsection}, we give preliminaries about Renner monoids and their characters. We define the rook monoid and symplectic Renner monoid, give Solomon's results on the A- and B-matrices of type $A$, and show that the B-matrix encodes multiplicities of certain restrictions.

In Section \ref{heckecharacterssection}, we develop these results for the Hecke algebra. In Section \ref{heckebmatrix}, we define the Hecke algebra character table for a general Renner monoid by defining ``standard elements''. We then prove that the B-matrix for the Hecke algebra is the same as for the Renner monoid, which gives an algorithm to compute the Hecke algebra character table (Theorem \ref{heckealgebracharactertable}). In Section \ref{nonstandardelements}, we then show that the Hecke algebra character table can be used to compute character values at nonstandard elements.

In Section \ref{amatrix}, we shift focus over to the symplectic Renner monoid $RSp_{2n}$, and compute the A-matrix (Theorem \ref{amatrixcomputation}). The rows and columns of this matrix are indexed by either partitions or pairs of partitions; our formula involves binomial coefficients made from partition data. This echoes Solomon's formula for the rook monoid A-matrix, which is also in terms of similar binomial coefficients.

In Section \ref{restrictionsection}, we give a Pieri rule for type $C$, which allows us to compute the B-matrix for $RSp_{2n}$ (Corollary \ref{bmatrixcomputation}). Our formula involves Littlewood-Richardson coefficients. Note that for $R_n$, the A-matrix is a $\{0,1\}$-matrix, while for $RSp_{2n}$ in general, it isn't. We close that section with a formula of group characters analogous to Solomon's type-A formula in \cite[Corollary~3.14]{solomon2002representations}. The important feature of the formula is that it is entirely a statement involving group characters, but a proof without monoids is not obvious.

Finally, in Section \ref{typecheckechartable}, we use our preceding work to compute character tables for the Hecke algebras of $R_n$ and $RSp_{2n}$. In Section \ref{charactertablerformulassection}, we extend character table formulas by Starkey and Shoji to Renner monoid Hecke algebras (see Theorem \ref{starkeyshojimonoid}). Then in the next subsection, we perform some example computations, including the Hecke algebra character table of $RSp_4$.

In Section \ref{furtherquestionssection}, we discuss some open questions suggested by this research. Then in the appendix, we compute the A- and B-matrices, and thus the character table, for $RSp_6$.

\textbf{Acknowledgements:} This research was primarily conducted at the 2018 University of Minnesota-Twin Cities REU in combinatorics and algebra. Our research was supported by NSF RTG grant DMS-1745638. We would like to thank Benjamin Brubaker for his help, support, and mentorship throughout the project. We would also like to thank the anonymous referee for a careful reading of the paper and for suggestions that lead to multiple improvements in the exposition.

\section{Preliminaries on Renner monoids} \label{rennermonoidsection}

In this section, we set up notation and cite results from other papers about the structure of Renner monoids and their characters. A particular focus for us is the symplectic Renner monoid $RSp_{2n}$, which we define in Section \ref{rspdefinitionsection}. Our main tool in this paper is the analogue of Solomon's A- and B-matrices, which are matrices which encode a Renner monoid character table in terms of the (group) character tables of certain subgroups $W_e$; we define these matrices in Section \ref{decomposingchartable}.

\subsection{Renner Monoids}

We lay out some important definitions relating to the Renner monoid. See \cite{liRennerMonoid} for more details.

Let $\mathcal{M}$ be a finite reductive monoid, let $G$ be its group of units, and let $R$ be its Renner monoid. Let $W$ be the Weyl group of $G$, with simple reflections $S$, and let $\Lambda$ be the cross-section lattice of idempotents of $R$. Every idempotent of $R$ is conjugate to a unique element of $\Lambda$. For $e\in\Lambda$, define \[W(e) = \{w\in W \;|\; we = ew\}, \hspace{30pt} W_*(e) = \{w\in W \;|\; we = ew = e\},\] \[W^*(e) = W(e)/W_*(e).\] In other words, $W^*(e)$ is the parabolic subgroup of $W$ generated by the set of simple reflections $s\in S$ where $es = se \ne e$. We have $W(e) = W^*(e)\times W_*(e)$.

We can view $W^*(e)$ as sitting in $R$: $eW^*(e)e$ is a subsemigroup of $R$ that is a group with identity $e$, and thus is isomorphic as a group to $W^*(e)$. We set $W_e:=eW^*(e)e$ for convenience. Note that \[W_e = eW^*(e)e = eW^*(e) = W^*(e)e = eW(e)e = eW(e) = W(e)e,\] since $eW_*(e) = W_*(e)e = e$.

For $e,f\in\Lambda$, and for any $w\in W$, the cosets $W(e)w$, $wW(e), W_*(e)w, wW_*(e)$, and $W(e)wW(f)$ have unique elements of minimal length. Let the sets of these minimal length elements be denoted Red$(e,\cdot)$, Red$(\cdot,e)$, Red$_*(e,\cdot)$, Red$_*(\cdot,e)$, Red$(e,f)$, respectively. So for example, if $w\in \text{Red}(e,\cdot)$, then if $ew = w_1ew_2$, we must have $l(w_2)\ge l(w)$, and if $w\in \text{Red}_*(e,\cdot)$, then if $ew = w_1ew_2$, we must have $l(w_1) + l(w_2)\ge l(w)$.

\begin{proposition} \label{uniqueidempotentproposition} \cite[Proposition~1.11]{godelle}
Let $r\in R$. Then,
\begin{itemize}
\item There exists a unique triple $(w_1,e,w_2)$ such that $r = w_1ew_2$ and $e\in\Lambda$, $w_1\in \text{Red}_*(\cdot,e), w_2\in \text{Red}(e,\cdot)$.
\item There exists a unique triple $(w_1',e,w_2')$ (with the same $e$) such that $r = w_1'ew_2'$ and $w_1\in \text{Red}(\cdot,e), w_2\in \text{Red}_*(e,\cdot)$.
\end{itemize}
\end{proposition}

In particular, to every $r\in R$ we can associate a unique idempotent $e\in\Lambda$. We also obtain a length function: $l(r) = l(w_1) + l(w_2) = l(w_1') + l(w_2')$.

\subsection{Character Theory of Renner Monoids}

In this section, we summarize some facts from Li, Li, and Cao's work on representations of Renner monoids \cite{liRennerMonoid} that are salient to us.

Let $\chi$ be an irreducible character of $R$. Then $\chi$ is constant on the elements of the quotient $R/[R,R]$, where $[R,R]$ is the submonoid of commutators of $R$. We will call elements of $R/[R,R]$ the \textit{Munn classes} of $R$, so we can say that $\chi$ is determined by its values on representatives of the Munn classes. The character table of $R$ is square, with both rows and columns indexed by the set \[\mathcal{Q}(R) = \bigsqcup_{e\in\Lambda} \mathcal{P}(W_e), \hspace{20pt} \text{where } \mathcal{P}(W_e) \text{ indexes the representations of } W_e.\]

Given $\lambda\in \mathcal{P}(W_e)\subset\mathcal{Q}(R)$, the Munn class corresponding to $\lambda$ can be represented by $r_\lambda = ew_\lambda e$, where $w_\lambda$ is a representative for the conjugacy class of $W_e$ corresponding to $\lambda$ (which we also denote $W_\lambda$). The character $\chi_\lambda$ can be written as a sum of character values of $\ba{\chi}_\lambda$, where $\ba{\chi}_\lambda$ is the character of $W_e$ corresponding to $\lambda$ \cite[Theorem~4.1]{liRennerMonoid}. In particular, if $\lambda\in\mathcal{P}(W)$, then $\chi(w) = \ba{\chi}(w)$ for all $w\in W$.

In fact, also by \cite[Theorem~4.1]{liRennerMonoid}, we have:

\begin{proposition} \label{charactertableuppertriangular}
Let $M:=M(R)$ be the character table of $R$, let $\lambda\in\mathcal{P}(W_e), \mu\in\mathcal{P}(W_f)$. Then if $e>f$, $\chi_\lambda(r_\mu) = 0$, and if $e=f$, $\chi_\lambda(r_\mu) = \ba{\chi}_\lambda(w_\mu)$.
\end{proposition}

We pick an order of $\mathcal{Q}$ once and for all such that if we index the rows and columns of $M$ with this order, then if $\lambda\in\mathcal{P}(W_e), \mu\in\mathcal{P}(W_f)$ with $e<f$ then the row corresponding to $\lambda$ is higher than the row corresponding to $\mu$.

\subsection{Injective Partial Transformations and the Rook Monoid} \label{rookmonoiddefinitionsection}

The \textit{rook monoid} $R_n$ is the monoid of injective partial transformations of $\mathbf{n} = \{1,2,\ldots,n\}$. It is the Renner monoid of type $A_{n-1}$. We follow notation in \cite{li2008representations,liRennerMonoid}.

An injective partial transformation of $\mathbf{n}$ is a bijective map $\sigma:I(\sigma)\to J(\sigma)$ where $I(\sigma),J(\sigma)\subseteq \mathbf{n}$. When $I(\sigma)=\mathbf{n}$, $\sigma$ is a permutation of $\mathbf{n}$.

We compose injective partial transformations (from the left) by viewing them as maps $\mathbf{n}\cup \{0\}\to\mathbf{n}\cup \{0\}$ where every element of $\mathbf{n}\cup \{0\}$ not in $I(\sigma)$ maps to 0. We can also view an injective partial transformations as a partial permutation matrix i.e. an $n\times n$ matrix with at most one 1 per row and column, with all other entries 0. With this identification, composition of functions corresponds to matrix multiplication.

Set $I^\circ(\sigma) = \{i\in\mathbf{n} \; | \; \sigma^k(i) \ne 0 \text{ for all } k\ge 0\}$, let $\sigma^\circ$ be the partial injective transformation \[\sigma^\circ(i) := \begin{cases} \sigma(i), & \text{if } i\in I^\circ(\sigma) \\ 0,& \text{else},\end{cases}\] and let $\sigma_\circ = \sigma - \sigma^\circ$, where the subtraction is subtraction of matrices. Then, $\sigma^\circ$ is a permutation on the set $I^\circ(\sigma)$, and $\sigma_\circ$ is nilpotent.

Given $K\subseteq\mathbf{n}$ with $|K|=t$, define the idempotent $e_K$ to be the identity transformation on $K$. Let $\mu_K$ (resp. $\mu^-_K$) be a bijective, order-preserving map from $\mathbf{t}$ to $K$ (resp. from $K$ to $\mathbf{t}$). These elements can be seen as rank $|K|$ ``inverses'': $\mu_K\mu^-_K = e_K$ and $\mu^-_K\mu_K = e_{\mathbf{t}}$. Given $\sigma\in W$, define $\sigma_K:= \mu^-_K\sigma\mu_K$; $\sigma_K$ permutes the entries of $\mathbf{t}$ and is 0 elsewhere, so we can view it as a permutation in $S_t$. (Note that our permutations act on the left, while in \cite{li2008representations}, they act on the right, so their definition of $\sigma_K$ is the reverse of ours).

The group of units of $R_n$ is the Weyl group $W = S_n$. The cross-sectional lattice $\Lambda$ is $\{e_t, 0\le t\le n\}$. The groups $W_{e_t}$ are $S_t$, where $S_0$ is taken to be the trivial group. The set $\mathcal{P}_t:=\mathcal{P}(S_t)$ is the set of partitions of $t$, and therefore, $\mathcal{Q}(R_n) = \{\lambda \;|\; \lambda\vdash t, 0\le t\le n\}$.

\subsection{The Symplectic Renner Monoid} \label{rspdefinitionsection}

The Renner monoid of type $C_n$ is called either the \textit{symplectic Renner monoid} or the \textit{symplectic rook monoid}. It is the set of symplectic injective partial transformations, and is denoted $RSp_{2n}$. We follow \cite{li2008representations} for its definition.

Let $m=2n$, and define the bar involution on $\mathbf{m}$ to be given by $\ba{i} = m+1-i$. We say that a subset $P$ of $\mathbf{m}$ is admissible if $P=\mathbf{m}$ or if for all $i\in P, \ba{i}\notin P$. We say that a bijective transformation on $\mathbf{m}$ is symplectic if it sends admissible sets to admissible sets, and we say that an injective partial transformation on $\mathbf{m}$ is symplectic if it is bijective and symplectic or if both its domain and range are proper admissible sets. In particular, a cycle is symplectic if and only if the set of its entries is admissible. We define: \[WSp_{2n} = \{\text{symplectic bijective transformations on } \mathbf{m}\},\] \[RSp_{2n} = \{\text{symplectic injective partial transformations on } \mathbf{m}\}.\]

If $\sigma\in RSp_{2n}$ and $K$ is an admissible set, the elements $\mu_K,\mu^-_K$, and $\sigma_K$ are all symplectic injective partial transformations, so are all contained in $RSp_{2n}$. For any $0\le t\le 2n$, we define the set \[C(t,\sigma) := \{K\subseteq I^\circ(\sigma) \; | \; K \text{ admissible},\; |K|=t,\; \sigma(K)=K\}.\] Then $C(t,\sigma)$ represents all of the admissible subsets of size $t$ which consist of all the entries some cycles of the permutation $\sigma^\circ$. $\sigma$ acts on elements of $C(t,\sigma)$ as permutations in $S_t$, so we also let \[S_{\mu,\sigma} = \{K\in C(t,\sigma) \;|\; |\mu|=t \text{ and } \sigma_K \text{ has cycle type } \mu\}.\] If $\sigma\in R_n$, we view $\mathbf{n}$ as a subset of $\mathbf{m}$ by inclusion, and since every subset of $\mathbf{n}$ is admissible, we get \[C(t,\sigma) = \{K\subseteq I^\circ(\sigma) \; | \; |K|=t, \; \sigma(K)=K\}.\] These definitions will be important in Section \ref{amatrix}.

The group of units of $RSp_{2n}$ is the Weyl group $W=WSp_{2n}$. We give a matrix description of this monoid in Section \ref{SymplecMatrix}. The cross-sectional lattice $\Lambda$ is the set $\{e_t,\; 0\le t\le n\}\cup\{1\}$. The groups $W_e,\; e\in \Lambda,\; e\ne 1$ are the groups $S_t,\; 0\le t\le n$. Therefore, \[\mathcal{Q}(RSp_{2n}) = \mathcal{SP}_n \sqcup \bigcup_{t=0}^n \mathcal{P}_t,\] where $\mathcal{SP}_n:=\mathcal{P}(WSp_{2n})$ is the set of double partitions of total size $n$ \cite[Theorem~5.5.6]{geck2000characters}.

Our main tool when working with $RSp_{2n}$ will be the character formula of Li, Li, and Cao:

\begin{proposition} \label{symplecticcharactertable} \cite[Theorem~4.2]{li2008representations}
For all $\sigma\in RSp_{2n}$, if $\chi$ is a character of $RSp_{2n}$ such that $\ba{\chi}$ is a character of $S_t$, then \[\chi(\sigma) = \sum_{K\in C(t,\sigma)} \ba{\chi}(\sigma_K).\]
\end{proposition}

Note that $RSp_{2n}$ is also the Renner monoid of type $B_n$ by \cite[Theorem~12]{liAlgebraicMonoid}.

\subsection{Matrix Description for the Symplectic Renner Monoid} \label{SymplecMatrix}
Li, Li, and Cao show in \cite[Proposition~2.1]{li2008representations} that
\[RSp_{2n} = \{A \in R_{2n} \; | \; A P A^t = A^t P A = 0 \} \cup WSp_{2n}\]
where 
\[P = \begin{pmatrix}
0 & J_n \\
-J_n & 0
\end{pmatrix}, \qquad J_n = \begin{pmatrix} 0 & \dots &  & & 1 \\
\dots & &  & 1 &  \\
& & \dots &  & \\
& 1 &  &  & \dots \\
1 &&& \dots & 0\end{pmatrix}.\]

$J_n$ is the $n \times n$ anti-diagonal matrix of $1$'s, and we view $R_{2n}$ as a submonoid of $\text{Mat}_{2n}$. $WSp_{2n}\subset R_{2n}$, so $WSp_{2n}$ is a subgroup of $S_{2n}$. From this, we establish:

\begin{proposition}
\label{SymplecRook}
\[
RSp_{2n} \cong \{ A \in R_{2n} \; | \; A^t J_{2n} A = A J_{2n} A^t = 0 \text{ or } J_{2n} \}
\]
\end{proposition}
\begin{proof}
 We first show that
\[\{A \in R_{2n} \; | \; A^t P A = A P A^t = 0\} = \{A \in R_{2n} \; | \; A^t J_{2n} A = A J_{2n} A^t = 0\}\]
Indeed, note that $J_{2n} A^t = (b_{ij})$ is still an element of $R_{2n}$ by nature of having at most one entry in every row and column that is equal to $1$, so that the condition that
$$A J_{2n} A^t = 0 \iff \quad \forall i,j, \quad (A J_{2n} A^t)_{ij} = c_{ij} = \sum_{k = 1}^{2n} a_{ik} b_{kj} = 0 $$
means all the $a_{ik}b_{kj}$ are 0 (since all summands must be non-negative in the evaluation of $c_{ij}$). Clearly, if $d_{ij} := (A P A^t)_{ij}$, then $c_{ij} = \pm d_{ij}$, so that  $c_{ij} = 0 \iff d_{ij} = 0$. Thus $A J_{2n} A^t = 0 \iff A P A^t = 0$, and the non-full rank components of these sets coincide.

It suffices to show that
$$WSp_{2n} = \{A \in R_{2n} \; | \; A^t J_{2n} A = A J_{2n} A^t = J_{2n} \}$$
For $A \in WSp_{2n}$, we have that
$$A(k) = i_k \implies A(\overline{k}) = \overline{i_k}$$
for if not, then using our first definition of the symplectic rook monoid, $\{k, A^{-1}(\overline{i_k})\}$ would be an admissible set mapped to a non-admissible set. Yet note that
$$A^t J_{2n} A = A J_{2n} A^t = J_{2n} \iff A J_{2n} = J_{2n} A \iff J_{2n} A J_{2n} = A$$
having used the fact that $J_{2n}^{-1} = J_{2n}$ and $A^t = A^{-1}$ which follows because $A \in W$ must be full rank and hence orthogonal. However, $J_{2n} A J_{2n} = A$ is exactly the condition that $A(k) = i_k \iff A(\overline{k}) = \overline{i_k}$, as $J_{2n}$ is the permutation corresponding to $(1 \; \overline{1}) (2 \; \overline{2}) \cdots (n \; \overline{n})$ and hence conjugating a permutation matrix by $J_{2n}$ makes it so that $k \mapsto \overline{A(\overline{k})}$. Thus,
$$A = J_{2n} A J_{2n} \iff A(k) = \overline{A(\overline{k})} \iff \overline{A(k)} = A(\overline{k})$$
so every $A \in WSp_{2n}$ satisfies $A = J_{2n} A J_{2n}$. Similarly, matrices $A \in R_{2n}$ satisfying $A = J_{2n} A J_{2n}$ are full rank and map admissible sets to admissible sets. For if not, then there would exist
$$k, \; s \neq \overline{k} \; \text{ s.t. } \; A(k) = i_k, \quad A(s) = \overline{i_k} $$
contradicting the condition that $J_{2n} A J_{2n} = A$. 
\end{proof}

\begin{remark}

Li, Li, and Cao \cite[Corollary~2.3]{li2008representations} suggest an alternate equivalent description of the symplectic rook monoid:
$$ RSp_{2n} = \{A \in R_{2n} \; | \; A P A^t = A^t P A = 0 \text{ or } P\}.$$
Yet note that \[WSp_2 = \left\{\begin{pmatrix} 1&0\\0&1\end{pmatrix}, \begin{pmatrix} 0&1\\1&0\end{pmatrix}\right\},\] but
$$A = \begin{pmatrix} 
0 & 1 \\
1 & 0
\end{pmatrix}, P = \begin{pmatrix} 
0 & 1 \\
-1 & 0
\end{pmatrix} \implies A P A^t = \begin{pmatrix}
0 & -1 \\
1 & 0
\end{pmatrix} \ne 0 \text{ or } P,$$
which does not match the original definition. Our Proposition \ref{SymplecRook} is intended as a fix for this minor error. Note that this does not pose a significant problem for that paper, as the authors mostly use their definition in terms of admissible sets.

\end{remark}

\subsection{Decomposing Character Tables of Finite Inverse Semigroups} \label{decomposingchartable}

Steinberg \cite{steinberg} tells us that the character table $M$ of any finite inverse semigroup is invertible. Therefore, we can define the Solomon A- and B-matrices:

\begin{definition}
Let $M:=M(R)$ be the character table of the Renner monoid $R$. Let $Y:=Y(R)$ be the block diagonal matrix with blocks $M(W_e)$ for all $e\in\Lambda$, with rows and columns indexed the same way as $M$. Then we define the invertible matrices $A$ and $B$ as \[M = AY = YB.\]
\end{definition}

In type $A$, Solomon showed that these matrices are combinatorially interesting. For any partition $\lambda$, let $\lambda_{(i)}$ denote the number of parts of $\lambda$ of size $i$. For two partitions $\lambda,\mu$, let \[\binom{\lambda}{\mu} := \prod_{i\ge 1} \binom{\lambda_{(i)}}{\mu_{(i)}}.\]

\begin{proposition} \label{typeamatrices} \cite[Propositions~3.5,~3.11]{solomon2002representations}
Let $R=R_n$, $\lambda\in \mathcal{P}_l, \mu\in\mathcal{P}_m$. Then \[A_{\lambda,\mu} = \binom{\lambda}{\mu}, \hspace{40pt} B_{\lambda,\mu} = \begin{cases} 1, & \lambda\supseteq \mu, \text{ and $\lambda-\mu$ is a horizontal strip} \\ 0, & \text{otherwise}. \end{cases}\]
\end{proposition}

We will obtain similar (although more complicated) combinatorial descriptions for the A- and B-matrices in the case where $R = RSp_{2n}$ (Theorem \ref{amatrixcomputation}, Corollary \ref{bmatrixcomputation}). 

The B-matrix has a representation theoretic interpretation, due in type $A$ to Solomon in the proof of \cite[Proposition~3.11]{solomon2002representations}:

\begin{proposition} \label{bmatrixmultiplicities}
Let $R$ be a Renner monoid, with B-matrix $B$. Then \[B_{\lambda,\mu} = \dim\Hom\left(Res^R_{W_\lambda} \chi_\mu,\; \ba{\chi}_\lambda\right),\] the multiplicity of $\chi_\lambda$ in $Res^R_{W_\lambda} \chi_\mu$.
\end{proposition}

\begin{proof}
Since $M=YB$, for any $\nu\in\mathcal{Q}(R)$, \[\chi_\mu(r_\nu) = \sum_{\lambda\in\mathcal{Q}(R)} Y_{\nu\lambda} B_{\lambda\mu} = \sum_{\lambda\in\mathcal{P}(W_\nu)} \ba{\chi}_\lambda(w_\nu) B_{\lambda\mu}.\] On the other hand, \[\chi_\mu(r_\nu) = \left(\text{Res}^R_{W_\nu} \chi_\mu\right)(w_\nu) = \sum_{\lambda\in\mathcal{P}(W_\nu)} \ba{\chi}_\lambda(w_\nu) \dim\Hom\left(\text{Res}^R_{W_\nu} \chi_\mu, \; \ba{\chi}_\lambda\right).\]

Now, by linear independence of characters, and noting that $W_\nu = W_\lambda$, the result follows.
\end{proof}

Our formulas for the A- and B-matrices for $RSP_{2n}$ in Sections \ref{amatrix} and \ref{restrictionsection} give two combinatorial interpretations of the $RSp_{2n}$ character table.

\section{Characters of Hecke Algebras} \label{heckecharacterssection}

In this section, we give a procedure to compute the Renner monoid Hecke algebra character table for any Renner monoid $R$. This method uses the B-matrix of $R$, as well as the character tables of the $W_e$. Dieng, Halverson, and Poladian \cite{Dieng2003} have computed the character table for the Renner monoid Hecke algebra in Type $A$. In Section \ref{typecheckechartable}, we do the same for type $C$.

Our method for computing the character tables for Renner monoids has been the Solomon decomposition \[M = AY = YB.\] In this section, we consider the Hecke algebra version of this decomposition, \[M_q = Y_qB_q.\] $B$ encodes multiplicities of restrictions, which we will use to show that $B_q=B$. Since $M$ and $Y_q$ are known for all types, we can compute $B$, and thus the character table of any Renner monoid Hecke algebra by a simple matrix multiplication. 

In Section \ref{heckebmatrix}, we prove the formula $M_q = Y_qB$, and in Section \ref{nonstandardelements}, we show that the character values of every element of the Hecke algebra can be given as a linear combination of the values on so called ``standard elements''. Later on, in Section \ref{typecheckechartable}, we use our method to explicitly compute the character table of $\mathcal{H}(RSp_4)$.

\subsection{Character Values on Standard Elements of $\mathcal{H}(R)$: The B-Matrix}\label{heckebmatrix}

Let $R$ be a Renner monoid, and let $\mathcal{H}(R)$ be the (one-parameter) generic Hecke algebra of $R$ in the sense of \cite[Definition~1.30]{godelle}. In particular, for some finite algebraic extension $K$ of $\C(q)$, $\mathcal{H}(R)$ is the $K$-algebra with basis $\{T_r\;|\;r\in R\}$ and relations \[T_xT_r = \begin{cases} T_{xr}, & \text{if } x\in S,\; l(xr) = l(r)+1 \\ qT_r, & \text{if } x\in S,\; l(xr) = l(r) \\ (q-1)T_r + qT_{xr}, & \text{if } x\in S,\; l(xr) = l(r)-1 \\ q^{l(r)-l(xr)}T_{xr}, & \text{if } x\in \Lambda.\end{cases}\]

Since $W_e$ is isomorphic (as a group) to $W^*(e)$, $\mathcal{H}(W_e)$ includes into $\mathcal{H}(R)$ via $T_w\mapsto T_{ew}$. Clearly this is not the only possible inclusion, as $T_w\mapsto hT_{ew}h^{-1}$ is also an inclusion for any invertible element $h\in\mathcal{H}(R)$.

By the Putcha-Tits deformation theorem \cite[p.~347]{solomon1995introduction}, $\mathcal{H}(R)$ has the same representation theory as $R$. Every irreducible character $\chi$ of $R$ corresponds to an irreducible character of $\mathcal{H}(R)$, and we write $\chi^*$ for this character. Thus we see that the B-matrices of the Renner monoid and Hecke algebra are equal:

\begin{proposition} \label{heckealgebramultiplicities}
$B_q = B$, the B-matrix of $R$.
\end{proposition}

\begin{proof}
By Proposition \ref{bmatrixmultiplicities}, $B_{\lambda,\mu} = \dim\Hom\left(Res^R_{W_\lambda} \chi_\mu,\ba{\chi}_\lambda\right)$. Note that the proof of Proposition \ref{bmatrixmultiplicities} applies also to $\mathcal{H}(R)$, so \[(B_q)_{\lambda,\mu} = \dim\Hom\left(Res^{\mathcal{H}(R)}_{\mathcal{H}(W_\lambda)} \chi^*_\mu,\; \ba{\chi}^*_\lambda\right)\]

Since $\mathcal{H}(R)\cong K[R]$, these multiplicities are the same, so \[B_{\lambda,\mu} = \dim\Hom\left(\Res^R_{W_\lambda} \chi_\mu,\; \ba{\chi}_\lambda\right) = \dim\Hom\left(\Res^{\mathcal{H}(R)}_{\mathcal{H}(W_\lambda)} \chi^*_\mu,\; \ba{\chi}^*_\lambda\right) = (B_q)_{\lambda,\mu}.\]
\end{proof}

For $\mu\in\mathcal{Q}$, let $\ba{\chi}_\mu^*$ be the character of the group Hecke algebra $\mathcal{H}(W_e)$ corresponding to $\mu$.

\begin{cor} \label{heckealgebramultiplicitiescorollary}
If $h\in\mathcal{H}(W_e)\subset\mathcal{H}(R)$, then \[\chi_\lambda^*(h) = \Res^{\mathcal{H}(R)}_{\mathcal{H}(W_e)}\chi^*(h) = \sum_{\mu\in\mathcal{P}(W_e)} B_{\mu\lambda} \ba{\chi}_\mu^*(h).\] \qed
\end{cor}

The rows of the character table for $R$ are indexed by Munn classes; we want to create a Hecke algebra character table indexed by these same Munn classes. We do this by choosing row representatives, which we call ``standard elements'', such that the Hecke algebra character table becomes the monoid character table under the specialization $q\mapsto 1$. Moreover, we want these representatives to be of the form $T_\lambda := T_{r_\lambda}$, where $r_\lambda\in R$ is an element of the Munn class $\lambda$. This construction is done for group Hecke algebras in \cite[Section~8.2]{geck2000characters}.

Let $\lambda\in\mathcal{Q}(R)$. Notice that from Corollary \ref{heckealgebramultiplicitiescorollary}, $T_\lambda\in\mathcal{H}(R)$ will be a standard element if it is a standard element in $\mathcal{H}(W_\lambda)$. By \cite[Definition~8.2.9]{geck2000characters}, it is sufficient to take $r_\lambda = ewe$, where $w\in W_\lambda = W_e$ is any element of minimal length in its conjugacy class (because $W_\lambda$ is a parabolic subgroup of $W$, this element will indeed be in $W_\lambda$).

We summarize the above work as follows:

\begin{definition} \label{charactertabledefinition}
Let $R$ be a Renner monoid with Hecke algebra $\mathcal{H}(R)$. Then the character table of $\mathcal{H}(R)$ is the matrix $M_q$, where for $\lambda,\mu\in\mathcal{Q}(R)$, \[(M_q)_{\lambda\mu} = \chi_\mu^*(T_\lambda).\]
\end{definition}

\begin{theorem} \label{heckealgebracharactertable}
Let $Y_q$ be the block diagonal matrix with blocks $M(\mathcal{H}(W_e))$ for all $e\in\Lambda$, and let $B$ be the B-matrix of $R$. Then \[M_q = Y_qB.\] \qed
\end{theorem}

Theorem \ref{heckealgebracharactertable} gives an algorithm to compute the Hecke algebra character table for any Renner monoid. The character table for the group Hecke algebra has been computed for every type, with important work done by Starkey, Pfeiffer, Geck, and others \cite[\S 9-11]{geck2000characters}; so $Y_q$ can be obtained. In addition, Li, Li, and Cao \cite{liRennerMonoid} have computed the character table for any Renner monoid, and so the formula $B = Y^{-1}M$ can be used to work backwards and obtain the $B$-matrix in the cases where an explicit formula is not already known.

The next subsection shows that we can compute every character value of $\mathcal{H}(R)$ from the character table, and the remaining sections are devoted to enacting our program for the type $C$ Renner monoid $RSp_{2n}$.

\subsection{Non-Standard Elements} \label{nonstandardelements}

In this section, we prove that Definition \ref{charactertabledefinition} is truly a character table, in the sense that we can obtain the character values on any element of $\mathcal{H}$ by taking linear combinations of the rows of the character table.

In other words, we prove the following:

\begin{theorem} \label{nonstandardelementstheorem}
Let $r\in R$, and let $\chi^*$ be an irreducible character of $\mathcal{H}(R)$. Then $\chi^*(T_r)$ can be expressed as a linear combination \[\chi^*(T_r) = \sum_{\lambda\in\mathcal{Q}(R)} c_{r,\lambda} \chi^*(T_\lambda)\] of the character values $\chi^*(T_\lambda)$, where the coefficients $c_{r,\lambda}$ do not depend on $\chi$.
\end{theorem}

Dieng, Halverson, and Poladian used a recursive argument to prove this result for type $A$ in \cite{Dieng2003}, and we use a similar argument here. For $h,h'\in\mathcal{H}(R)$, we say that $h$ and $h'$ are equivalent and write $h\equiv h'$ if $h$ and $h'$ become equal in the quotient $\mathcal{H}(R)/[\mathcal{H}(R),\mathcal{H}(R)]$. Since characters are trace functions, $h\equiv h'$ means that $\chi(h) = \chi(h')$ for all characters of $\mathcal{H}(R)$.

Then, Theorem \ref{nonstandardelementstheorem} follows from the following result.

\begin{proposition} \label{stdeltlinearcombination}
Every basis element $T_r$ of $\mathcal{H}$ is equivalent to a linear combination of standard elements.
\end{proposition}

We will prove this proposition with the help of a lemma.

\begin{lemma} \label{stdeltlemma}
Let $r = w_1ew_2\in R$, where $w_1,w_2\in W$, $e\in\Lambda$. Then $T_r$ is equivalent to a linear combination of elements of the form $T_{ewe}$, $w\in W$.
\end{lemma}

\begin{proof}
We have \[T_r = T_{w_1ew_2} = T_{w_1} T_e T_{w_2} \equiv T_e T_{w_2}T_{w_1} = \sum_{w\in W} a_w T_{ew},\] where the $a_w\in K$ arise from the Hecke algebra relations from multiplying $T_{w_2}$ and $T_{w_1}$.

Now, \[T_{ew} = T_{eew} = q^{l(ew)-l(w)}T_eT_eT_w \equiv q^{l(ew)-l(w)}T_eT_wT_e = q^{l(ew)-l(ewe)} T_{ewe},\] so \[T_r \equiv \sum_{w\in W} a_w q^{l(ew)-l(ewe)}T_{ewe}.\]
\end{proof}

\begin{proof}[Proof of Proposition \ref{stdeltlinearcombination}]
If $r = ewe$ for $w\in W(e)$, then $r\in W_e$, so $T_r\in\mathcal{H}(W_e)$, and Geck and Pfeiffer \cite[\S 8.2]{geck2000characters} give a procedure to write $T_r$ as equivalent to a linear combination of standard elements.

By Proposition \ref{uniqueidempotentproposition}, we can write $r = w_1ew_2$ for a unique idempotent $e\in\Lambda$, and we can choose $w_1,w_2\in W$ so that $l(r) = l(w_1) + l(w_2)$. We will reduce this to the case covered by Geck and Pfeiffer by induction on $\Lambda$ (note that $\Lambda$ is finite since $R$ is finite).

First, if $e$ is minimal, Lemma \ref{stdeltlemma} tells us that $T_r$ is equivalent to a linear combination of elements of the form $T_{ewe}$. By \cite[Corollary~1.13]{godelle}, since $e$ is minimal, we must have $w\in W(e)$.

Next, assume that $e$ is not minimal, and that for all $f\in\Lambda,\;f<e$, every element of the form $T_{xfz},\; x,z\in W$ is equivalent to a linear combination of standard elements. Again, Lemma \ref{stdeltlemma} tells us that $T_r$ is equivalent to a linear combination of elements of the form $T_{ewe}$.

By the definition of Red$(e,e)$, we can write $w = xyz$, where $x,z\in W(e),\; y\in \text{Red}(e,e)$, so $ewe = xeyez$. Godelle \cite[Corollary~1.13]{godelle} tells us that either $y\in W(e)$ or $eye=f$ for some $f\in \Lambda,\; f<e$. In the first case, $xyz\in W(e)$, so $ewe\in eW(e)e =  W_e$. In the second case, we have $ewe = xeyez = xfz$, so by hypothesis, $T_{ewe} = T_{xfz}$ is equivalent to a sum of standard elements.
\end{proof}

\begin{remark}
See Example \ref{nonstandardelementexample} for the computation of character values for a nonstandard element of $\mathcal{H}(RSp_4)$.
\end{remark}

\section{Decomposing the Character Table of $RSp_{2n}$: The A-Matrix} \label{amatrix}

Most of the rest of the paper concerns the symplectic Renner monoid $RSp_{2n}$. In this section, we compute the A-matrix of $RSp_{2n}$. This is an analogous result to Solomon's \cite[Proposition~3.5]{solomon2002representations}, which computed the A-matrix for the rook monoid $R_n$. Let $M := M(RSp_{2n})$ denote the character table of $RSp_{2n}$, and recall that the rows of $M$ (and therefore the rows of our related matrices, $A$ and $B$) are labelled by Munn classes of elements of $RSp_{2n}$ and the columns are labelled by the irreducible representations, both of which are indexed by $\mathcal{Q}_n = \mathcal{Q}(RSp_{2n})$.

\subsection{Structure of the A-Matrix} \label{amatrixstructure}

Since $\mathcal{Q}_n = \mathcal{SP}_n\sqcup \bigcup_{t=0}^n \mathcal{P}_t$, we can write by \cite[Theorem~4.1]{liRennerMonoid} \[M = \begin{pmatrix} M(WSP_{2n}) & * \\ 0 & M(R_n) \end{pmatrix},\] so since $Y$ is block diagonal, we can write \[A = \begin{pmatrix} Id & U \\ 0 & T \end{pmatrix},\] where $Id$ is an identity matrix and $T$ is the A-matrix of $R_n$ computed in \cite[Proposition~3.5]{solomon2002representations}.

Therefore, to compute $A$, we need only compute the entries of $U$. Thus, we only need consider the values of irreducible representations indexed by elements of $\bigcup_{t=0}^n \mathcal{P}_t$.

\begin{lemma} 
Let $\alpha\in\mathcal{Q}_n,\; \lambda\vdash t,\; 0\le t\le n$, and let $r = r_\alpha$. Then the matrix entry
\[A_{\alpha, \lambda} = |S_{\lambda, r}|.\]
\end{lemma} \label{admissiblecycletypelemma}

\begin{proof}
By Proposition \ref{symplecticcharactertable},
\[M_{\alpha, \lambda} = \chi_{\lambda}(r) = \sum_{K \in C(t, r)} \ba{\chi}_{\lambda}(r_K).\]
We have seen (Section \ref{rookmonoiddefinitionsection}) that \[C(t,r) = \bigsqcup_{\mu\vdash t} S_{\mu,r},\] and all permutations $r_K$, $K\in S_{\mu,r}$ have cycle type $\mu$. Therefore, \[M_{\alpha, \lambda} = \chi_{\lambda}(r) = \sum_{K \in C(t,r)} \chi_{\lambda}(r_K) = \sum_{\mu \vdash t} \sum_{K \in S_{\mu,r}} \ba{\chi}_{\lambda}(r_K) = \sum_{\mu \vdash t} |S_{\mu,r}| \; \ba{\chi}_{\lambda}(w_\mu)\]
Note that 
\[(A Y)_{\alpha, \lambda} = \sum_{\beta \in \mathcal{Q}_n} A_{\alpha, \beta} Y_{\beta, \lambda} = \sum_{\mu \vdash t}  A_{\alpha, \mu} Y_{\mu, \lambda} = \sum_{\mu\vdash t}  A_{\alpha, \mu} \ba{\chi}_{\lambda}(w_\mu),\]
where the second equality is because $Y$ is block diagonal.
Thus, by linear independence of characters of $S_t$, we see that $A_{\alpha, \mu} = |S_{\mu,r}|$.
 \end{proof}

\begin{ex} Lemma \ref{admissiblecycletypelemma} allows us to compute Solomon's A-matrix $T$.

In the case that $\alpha\in\mathcal{P}_t, \; 0 \le t \le n$, the representative $r_\alpha$ has a proper admissible domain $I(r_\alpha)$, so every subset $K\subset I(r_\alpha)$ is admissible. Hence, an element of $S_{\lambda,r_\alpha}$ is simply a choice of cycles of $r_\alpha$ with cycle type $\lambda$.

The number of ways to choose $\lambda_{(i)}$ cycles of length $i$ from $\alpha$ is $\binom{\alpha_{(i)}}{\lambda_{(i)}}$, so multiplying over all $i$, we get that \[T_{\alpha,\lambda} = \binom{\alpha}{\lambda},\] which matches Proposition \ref{typeamatrices}.
\end{ex}

\subsection{Determining the $U$ Block for $RSp_{2n}$}

All that remains to compute $A$ is a combinatorial description of $|S_{\lambda,r_\alpha}|$, where $\lambda\in \mathcal{P}_t$, \; $\alpha\in \mathcal{SP}_n$. In this case, $r_\alpha$ is a group conjugacy class representative $w_\alpha$. Our computation will use particularly nice conjugacy class representatives from \cite[Section~3.4]{geck2000characters}.

Let $s_i,\; 1\le i\le n-1$ be the permutation $(i,\; i+1)(\ba{i},\; \ba{i+1})$, and let $t_i = (i+1,\; \ba{i+1})$. Then $W_n$ is a Coxeter group with generators $\{t_0,s_1,\ldots,s_{n-1}\}$, and $t_i = s_is_{i-1}\ldots s_1t_0s_1\ldots s_{i-1}s_i$.

Let $\alpha\in\mathcal{SP}_n$. Then $\alpha$ is a double partition $(\gamma,\delta)$ with $|\gamma|+|\delta|=n$. Write $\delta = (a_1,\ldots,a_p)$ with $a_1\le a_2\ldots\le a_p$ and $\gamma = (a_{p+1},\ldots,a_q)$ with $a_{p+1}\ge a_{p+2}\ldots\ge a_q$.

Then, set \[w_\alpha := b^-_{m_1,a_1}\ldots b^-_{m_p,a_p} b^+_{m_{p+1},a_{p+1}}\ldots b^+_{m_q,a_q},\] where $m_j = \sum_{i=1}^{j-1} a_i$ and \[b^-_{m,a} = t_ms_{m+1}\ldots s_{m+a-1}, \hspace{30pt} b^+_{m,a} = s_{m+1}\ldots s_{m+a-1}.\] We call $b^+_{m,a}$ (resp. $b^-_{m,a}$) a positive (resp. negative) block.

By \cite[Proposition~3.4.7]{geck2000characters}, $\{w_\alpha\}$ is a complete set of (minimal length) representatives of the conjugacy classes of $W_n$.

\begin{lemma} \label{blockdescription}
We have the following description of the blocks $b^+_{m,a}$ and $b^-_{m,a}$: \[ b_{m,a}^+ = (m+1, \; m+2, \; \dots \; m+a) (\ba{m+1}, \; \ba{m+2}, \; \cdots \; \ba{m+a}),\] \[ b^-_{m,a} =(m+1, \; m+2, \cdots m+a-1, m+a, \; \overline{m+1}, \; \overline{m+2} \cdots \overline{m+a}),\]
\end{lemma}

\begin{proof}
These are straightforward computations from the definitions.
\end{proof}

\begin{ex}
\[b^+_{0,3} = s_1 s_2 = (1,\; 2)(\ba{1},\; \ba{2})(\ba{2},\; \ba{3}) = (1,\; 2,\; 3)(\ba{1},\; \ba{2},\; \ba{3}),\] while \[b^-_{0,3} = t_0 s_1 s_2 = (1,\; \ba{1}) (1,\; 2,\; 3)(\ba{1},\; \ba{2},\; \ba{3})  = (1,\; 2,\; 3,\; \ba{1},\; \ba{2},\; \ba{3}).\]
\end{ex}

Let $w:=w_\alpha$. If $K\in C(t,w)$, then $K$ must be admissible, so no cycle of $w$ supported in $K$ can contain both $i$ and $\ba{i}$ for any $i$. In other words, $K$ can only depend on the positive blocks of $w_\alpha$. Using this observation, we can now give a combinatorial description of the $|S_{\lambda,w}|$.

\begin{proposition}
\[|S_{\lambda, w}| = 2^{\sum_{i} \lambda_{(i)} } \cdot \binom{\gamma}{\lambda}.\]
\end{proposition}

\begin{proof}
By Lemma \ref{blockdescription}, only the cycles arising from positive blocks in $w_\alpha$ are admissible. Moreover $\gamma_{(i)}$ is the number of positive blocks of length $i$, and each positive block is comprised of two disjoint cycles.

To construct an element of $S_{\lambda,w}$, we must, for every $i$, specify $\lambda_{(i)}$ distinct cycles of length $i$ in $w$ in such a way that the union $K$ of their support is admissible; once we do this, $w_K$ will have cycle type $\lambda$. The admissibility condition in this case reduces to the condition that we only choose at most one of the two cycles in any given positive block. Multiplying over all $i$, \[|S_{\lambda, w}| = \prod_{i\ge 1} \binom{\gamma_{(i)}}{\lambda_{(i)}} \cdot 2^{\lambda_{(i)}} = 2^{\sum_i \lambda_{(i)}} \cdot\binom{\gamma}{\lambda}.\]
\end{proof}

We have now completed the computation of the A-matrix for $RSp_{2n}$.

\begin{theorem} \label{amatrixcomputation}
The A-matrix of $RSp_{2n}$ is
\[A = \begin{pmatrix}Id & U \\ 0 & T \end{pmatrix},\] where \[T_{\alpha,\lambda} = \binom{\alpha}{\lambda}, \hspace{30pt} \alpha\in \mathcal{P}_{t_1},\; \lambda\in \mathcal{P}_{t_2},\; 0\le t_1,t_2\le n,\] and \[U_{\alpha, \lambda} = 2^{\sum_{i} \lambda_{(i)} } \cdot \binom{\gamma}{\lambda}, \hspace{30pt} \alpha = (\gamma, \delta) \in \mathcal{SP}_n,\; \lambda\in \mathcal{P}_t, \; 0\le t\le n.\]
\end{theorem}

\section{Restricting Monoid Representations to Group Representations and Decomposing Character Tables} \label{restrictionsection}
In this section, we determine the B-matrix of $RSp_{2n}$ by restricting the irreducible representations to the group of units $WSp_{2n}$, using also Solomon's computation of the B-matrix of the rook monoid \cite[Proposition 3.11]{solomon2002representations}.

Similarly to the A-matrix, since \[M = \begin{pmatrix} M(WSP_{2n}) & * \\ 0 & M(R_n) \end{pmatrix},\] we must have \[B = \begin{pmatrix} Id & V \\ 0 & L \end{pmatrix},\] where $Id$ is an identity matrix, $L$ is the B-matrix of the rook monoid $R_n$, and $V$ is to be computed. Let $W_n:=WSp_{2n}$ for ease of notation.

\subsection{Restricting Representations From $RSp_{2n}$ to $W_n$}

\begin{definition}
Given groups $G$ and $H$, and corresponding representations $V_G$ and $V_H$, we define the box tensor representation $V_G \boxtimes V_H$ to be the representation of $G \times H$ with the action $(g,h) \cdot (v_1 \boxtimes v_2) = gv_1 \boxtimes hv_2$.
\end{definition}

We can now explain the restriction of an $RSp_{2n}$ representation to $W_n$, in analogue to Solomon's \cite[Corollary~3.15]{solomon2002representations} for the rook monoid. Solomon shows that given an irreducible character $\chi$ of $R_n$ corresponding to a partition of $t$, the restriction $\chi\mid_{S_n} = \Ind_{S_t \times S_{n-t}}^{S_n}(\ba{\chi} \boxtimes \eta_{n-t})$, where $\eta_{n-t}$ is the trivial representation on $S_{n-t}$. We show that similarly,

\begin{proposition} \label{pierirulegeneral}
If $\alpha\in \mathcal{P}_t, \; 0\le t\le n$, let $\chi:=\chi_{\alpha}$. Then 
\[
\chi|_{W_n} = \Ind_{S_t \times W_{n-t}}^{W_n}(\ba{\chi} \boxtimes \eta_{n-t}).
\]
\end{proposition}

\begin{proof}
Let $\sigma\in W_n$. By Proposition \ref{symplecticcharactertable}, \[\chi(\sigma) = \sum_{K \in C(\sigma,t)}\ba{\chi}(\sigma_K),\] where $\sigma_K = \mu^-_K\sigma\mu_K$. Now $\mu_K$ sends $\mathbf{t}$ to $K$ and $\mu_K^-$ does the reverse, so there exists an element $\tau:=\tau_K\in W_n$ such that $\tau|_{\mathbf{t}} = \mu_K$ and $\tau^{-1}|_K = \mu_K^-$. Then set $\sigma_K' := \tau^{-1}\sigma\tau$; this is an element of $W_n$ such that $\sigma_K'|_\mathbf{t} = \sigma_K$. Thus, $\ba{\chi}(\sigma_K) = (\ba{\chi}\otimes \eta_{n-t})(\sigma_K')$. Since $\mathbf{t}\cap K$ may be any subset of $\mathbf{t}$, the elements $\tau_K$, as $K$ runs over $C(\sigma,t)$, are a set of coset representatives of $W_n/(S_t\times W_{n-t})$. Hence, \[\chi(\sigma) = \sum_{K \in C(\sigma,r)} (\ba{\chi}\otimes \eta_{n-t})(\sigma_K') = \sum_{\tau \in W_n/(S_t\times W_{n-t})} (\ba{\chi}\otimes \eta_{n-t})(\tau\sigma\tau^{-1})  = \Ind_{S_t \times W_{n-t}}^{W_n}(\ba{\chi} \boxtimes \eta_{n-t})(\sigma).\]
\end{proof}

\begin{remark}
We expect this proof to extend to other Renner monoids, in particular type $D_n$, using \cite[Theorem~4.1]{liRennerMonoid} in place of \cite[Theorem~4.2]{li2008representations}. Likewise, the same proof yields Solomon's result for type $A_n$.
\end{remark}

\subsection{A Pieri Rule Analogue for the Type $C_n$ Weyl Group}\label{PieriC}
We'll now determine an analogue of the Pieri rule for $C_n$. We first highlight two facts from \cite{geck2000characters}.

\begin{proposition}\cite[Lemma~6.1.3]{geck2000characters} \label{inductionfact1}
Let $n \ge 1$ and $k,l \ge 0$ be integers such that $n = k + l$. Let $(\lambda_1, \lambda_2)$ and $(\mu_1,\mu_2)$ be pairs of partitions with $|\lambda_1| + |\lambda_2| = k$ and $|\mu_1| + |\mu_2| = l$. Then, using the diagonal embedding $W_k \times W_l \subseteq W_n$, we have 
\[\Ind_{W_k \times W_l}^{W_n}(\ba{\chi}_{a_1, a_2} \boxtimes \ba{\chi}_{b_1, b_2}) = \sum_{(\nu_1, \nu_2)}c_{a_1, b_1}^{\nu_1}c_{a_2, b_2}^{\nu_2}\chi_{\nu_1, \nu_2}
\] where the $c_{a,b}^\nu$ are Littlewood-Richardson coefficients and the sum runs over all pairs of partitions $(\nu_1, \nu_2)$ for which $|\nu_i| = |\lambda_i| + |\nu_i|$ for $i = 1, 2$.
\end{proposition}

\begin{proposition}\cite[Lemma~6.1.4]{geck2000characters} \label{inductionfact2}
Let $n \ge 1$ and consider the parabolic subgroup $S_n \subset W_n$. Let $\nu \vdash n$ and let $\ba{\chi}_\nu$ be the corresponding irreducible character. Then
\[
\Ind_{S_n}^{W_n}\ba{\chi}_{\nu} = \sum_{\lambda, \mu} c_{\lambda \mu}^{\nu}\ba{\chi}_{(\lambda, \mu)}.
\]
\end{proposition}

We can now derive a more explicit formula for $\Ind_{S_k \times W_l}^{W_n}(\ba{\chi}_{\nu} \boxtimes \eta_{l})$ for a fixed partition $\nu \vdash k$. Using transitivity of induction along with Propositions \ref{inductionfact1} and \ref{inductionfact2}, we have

\begin{align*}
    \Ind_{S_k \times W_l}^{W_n}(\ba{\chi}_\nu \boxtimes \eta_l) &= \Ind_{W_k \times W_l}^{W_n}\Ind_{S_k \times W_l}^{W_k \times W_l}(\ba{\chi}_\nu \boxtimes \eta_l)\\
    &= \Ind_{W_k \times W_l}^{W_n}\left(\sum_{\lambda, \mu} c_{\lambda, \mu}^{\nu}\ba{\chi}_{\lambda, \mu} \boxtimes \eta_l\right)\\
    &= \sum_{\lambda, \mu} c_{\lambda, \mu}^{\nu}\Ind_{W_k \times W_l}^{W_n}(\ba{\chi}_{\lambda,\mu} \boxtimes \ba{\chi}_{[l], \emptyset})\\
    &= \sum_{\substack{\lambda, \mu \\ \lambda + \mu \vdash k}}c_{\lambda, \mu}^{\nu}\sum_{\substack{\nu_1, \nu_2\\ |\nu_i| = |\lambda_i| + |\mu_i|\\ \nu_1 + \nu_2 \vdash n}}c_{\lambda, [l]}^{\nu_1}c_{\mu, \emptyset}^{\nu_2}\ba{\chi}_{\nu_1, \nu_2}
\end{align*}
where $[l]$ denotes a horizontal strip of size $l$.
Note that $c_{\mu, \empty}^{\nu_2} = 0$ unless $\mu = \nu_2$ in which case it is equal to $1$. This sum reduces to 
\begin{align*}
    \Ind_{S_k \times W_l}^{W_n}(\ba{\chi}_\nu \boxtimes \eta_l) &= \sum_{\substack{\lambda, \mu \\ \lambda + \mu \vdash k}}c_{\lambda, \mu}^{\nu}\sum_{\substack{\nu_1\\ \nu_1 + \mu \vdash n}}c_{\lambda, [l]}^{\nu_1}\ba{\chi}_{\nu_1, \mu}\\
    &= \sum_{\substack{\gamma, \mu\\ \gamma + \mu \vdash n}}\left(\sum_{\substack{\lambda,\, \lambda + \mu \vdash k\\\gamma - \lambda \text{ horiz. strip}\\  \text{of size l}}}c_{\lambda, \mu}^{\nu}c_{\lambda, [l]}^{\gamma}\right)\ba{\chi}_{\gamma, \mu}\\
    &= \sum_{\substack{\gamma, \mu\\ \gamma + \mu \vdash n}}\left(\sum_{\substack{\lambda\\\gamma - \lambda \text{ horiz strip}\\ \text{of size l}}}c_{\lambda, \mu}^{\nu}\right)\ba{\chi}_{\gamma, \mu}
\end{align*}
The last line comes from swapping the order of summation and noting that $c_{\lambda, [l]}^{\nu}$ is $1$ if $\gamma - \lambda$ is a horizontal strip and $0$ otherwise by \cite[Corollary~6.1.7]{geck2000characters}.

Thus we have found:
\begin{proposition} \label{bmatrixformula}
For $n \ge 1$ and integers $k,l$ such that $k + l = n$ we have that
 \[
 \Ind_{S_k \times W_l}^{W_n}(\ba{\chi}_\nu \boxtimes \eta_l)
 = \sum_{\substack{\gamma, \mu\\ \gamma + \mu \vdash n}}\left(\sum_{\substack{\lambda\\\gamma - \lambda \text{ horiz strip}\\ \text{of size l}}}c_{\lambda, \mu}^{\nu}\right)\ba{\chi}_{\gamma, \mu}
 \] \qed
\end{proposition}

Now we can determine the B-matrix for $RSp_{2n}$.

\begin{cor} \label{bmatrixcomputation}
In the matrix decomposition of 
\[
M(RSp_{2n}) = YB,
\]
with 
\[
B = \begin{pmatrix} Id & V \\ 0 & L\end{pmatrix},
\]
we have
\[
V_{(\gamma, \mu),\nu} = \sum\limits_{\substack{\lambda \text{ partition s.t.} \\ \gamma - \lambda \text{ horiz strip}\\ \text{of size $n-r$}}}c_{\lambda, \mu}^{\nu}
\]
where $(\gamma,\mu)\in \mathcal{SP}_n$, $\nu\in \mathcal{P}_r, 0\le r\le n$; and $L$ is the B-matrix of $R_n$ with coefficients
\[
L_{\gamma, \mu} = \begin{cases}
1 & \gamma - \mu \text{ is a horizontal strip} \\
0 & \text{ otherwise.}
\end{cases}
\]
\end{cor}
\begin{proof}
The values $L_{\gamma,\mu}$ are given by Proposition \ref{typeamatrices}, so we consider the case where $\gamma\in \mathcal{SP}_{2n}, \mu\in \mathcal{P}_t, 0\le t\le n$.

By Proposition \ref{bmatrixmultiplicities}, $B_{\gamma\mu}$ is the multiplicity of $\ba{\chi}_\gamma$ in $\text{Res}^{RSp_{2n}}_{W_n}\chi_\mu$. By Proposition \ref{pierirulegeneral}, \[\text{Res}^{RSp_{2n}}_{W_n}\chi_\mu = \text{Ind}_{S_t\times W_{n-t}}^{W_n} (\ba{\chi}_\mu \boxtimes \eta_{n-t}),\] and so Proposition \ref{bmatrixformula} gives the desired result.
\end{proof}

\subsection{A Formula of Group Characters}
We can use our calculation of the A- and B-matrices to determine a formula for group characters, in analogue to \cite[Corollary~3.14]{solomon2002representations}. Let $z_\alpha$ be the size of the centralizer of an element in $\alpha\in\mathcal{Q}_n$, and let $W$ be the diagonal matrix $(\delta_{\alpha,\beta} z_\alpha)_{\alpha,\beta}$. Then $YY^T=W$ by the column orthogonality relations of character tables. For the next result, we use the notation $\chi^\lambda:=\chi(w_\lambda)$ for any element $w_\lambda$ in the conjugacy class indexed by $\lambda$.

\begin{cor} \label{groupcharacterresult}
Let $\lambda = (\lambda_1,\lambda_2)\vdash n,\; \mu\vdash r\le n$. Then \[\sum_{\substack{\alpha = (\alpha_1,\alpha_2)\vdash n \\ \beta\vdash t}} z_\alpha^{-1} 2^{\sum_i \mu_i} \ba{\chi}_\alpha^\lambda \ba{\chi}_\beta^\mu \binom{\alpha_1}{\mu} = \sum_{\substack{\lambda\\\gamma - \lambda \text{ horiz strip}\\ \text{of size n-t}}}c_{\lambda, \mu}^{\nu}.\]
\end{cor}

\begin{proof}
By Corollary \ref{bmatrixcomputation}, the right side of this equation is equal to $B_{\gamma,\mu}$, so we will show that the left side also equals $B_{\gamma,\mu}$. Note that $B = Y^{-1} AY = Y^TW^{-1}AY$, so 
\begin{align*} B_{\lambda\mu} &= \sum_{\alpha} (Y^TW^{-1})_{\lambda,\alpha} (AY)_{\alpha,\mu} \\&= \sum_{\alpha} Y^T_{\alpha,\lambda} z^{-1}_\alpha (AY)_{\alpha,\mu} \\&= \sum_{\alpha}\sum_{\beta} z^{-1}_\alpha Y_{\alpha,\lambda} A_{\alpha,\beta}Y_{\beta\mu},\end{align*}
and the definition of $Y$ along with Theorem \ref{amatrixcomputation}, gives the result.
\end{proof}

\section{The Hecke Algebra Character Table for $RSp_{2n}$} \label{typecheckechartable}

\subsection{A Starkey-Shoji Rule} \label{charactertablerformulassection}

We computed the $B$ matrix of $RSp_{2n}$ in Corollary \ref{bmatrixcomputation}. The group Hecke algebra character tables for types $A$ and $C$ are well-known; see for instance, \cite[\S 10.2-10.3]{geck2000characters}. Therefore, we can use Theorem \ref{heckealgebracharactertable} to compute the character table $M(\mathcal{H}(RSp_{2n})) = Y_qB$ of $\mathcal{H}(RSp_{2n})$. The character table of $\mathcal{H}(R_n)$ appears in the lower right corner of the larger character table, so our computations in type $C$ include the type $A$ computations: the rook monoid Hecke algebra character table was first given in \cite{Dieng2003}.

Recall (Section \ref{heckecharacterssection}) that the irreducible characters and the standard elements of both $RSp_{2n}$ and $\mathcal{H}(RSp_{2n})$ are indexed by $\mathcal{Q}_n$. For $\lambda\in\mathcal{Q}_n$, $\chi_\lambda$, $\chi^*_\lambda$ are the characters of $RSp_{2n}$, $\mathcal{H}(RSp_{2n})$, respectively, while $r_\lambda$ (resp. $T_\lambda := T_{r_\lambda}$) denote the standard elements of the Renner monoid (resp. Hecke algebra). $w_\lambda$ and $T_{w_\lambda}$ denote the corresponding standard elements of $W_\lambda$ and $\mathcal{H}(W_\lambda)$, while the corresponding group characters are denoted $\ba{\chi}_\lambda, \ba{\chi}^*_\lambda$.

Then by Theorem \ref{heckealgebracharactertable}, \[\chi^*_{\lambda}(T_\mu) = (M_q)_{\mu,\lambda} = \sum_{\nu\in\mathcal{Q}_n} (Y_q)_{\mu,\nu} B_{\nu,\lambda} = \sum_{\nu\in\mathcal{Q}(W_\mu)} \ba{\chi}_\nu^*(T_{w_\mu}) B_{\nu,\lambda}.\]

In particular, by Corollary \ref{bmatrixcomputation}, we can break this up into several cases.

If $\lambda\in \mathcal{SP}_n,\; \mu\in \mathcal{P}_t$, then $\chi_\lambda(T_\mu) = 0$. The same is true if $\lambda\in \mathcal{P}_{t'}$ where $t'>t$.

If $W_\lambda = W_\mu$, then $\chi_\lambda(T_{\mu}) = \ba{\chi}_\lambda(T_{w_\mu})$. Hence, we have reduced the computation to two cases: $\lambda\in\mathcal{P}_t$, and either $\mu\in \mathcal{SP}_n$ or $\mu\in\mathcal{P}_{t'}$, $t'>t$.

Starkey's Rule \cite[Theorem~9.2.11]{geck2000characters} is a combinatorial rule for the character table of the type $A$ group Hecke algebra. Shoji \cite{shoji} generalized this formula to Ariki-Koiki algebras, including the type $C$ group Hecke algebra. We extend these formulas here to the Hecke algebras of type $A$ and type $C$ Renner monoids.

In the case $\lambda\in\mathcal{P}_t,\; \mu\in\mathcal{P}_{t'}$, we have \[\chi^*_\lambda(T_\mu) = \sum_{\nu\vdash t'} \ba{\chi}_\nu^*(T_{w_\mu}) B_{\nu,\lambda} = \sum_{\substack{\nu\vdash t'\\\nu-\lambda \text{ horizontal strip}}} \ba{\chi}_\nu^*(T_{w_\mu}).\]

Let $C_\alpha$ (resp. $S_\alpha$, $\rho_\alpha$) denote the conjugacy class (resp. parabolic subgroup, reflection representation) in $S_n$ corresponding to $\alpha$. Using Starkey's Rule \cite[Theorem~9.2.11]{geck2000characters}, our equation becomes:

\begin{align*} \chi^*_\lambda(T_\mu) &= \sum_{\substack{\nu\vdash t'\\\nu-\lambda \text{ horizontal strip}}} \sum_{\alpha\vdash t'} \ba{\chi}_\nu(w_\alpha) \frac{|C_\alpha\cap S_\mu|}{|S_\mu|} \det(q - \rho_\mu(w_\alpha)) \\&= \sum_{\alpha\vdash t'} \left(\sum_{\substack{\nu\vdash t'\\\nu-\lambda \text{ horizontal strip}}} \ba{\chi}_\nu(w_\alpha)\right) \frac{|C_\alpha\cap S_\mu|}{|S_\mu|} \det(q - \rho_\mu(w_\alpha)) \\&= \sum_{\alpha\vdash t'} \chi_\nu(r_\alpha) \frac{|C_\alpha\cap S_\mu|}{|S_\mu|} \det(q - \rho_\mu(w_\alpha)),\end{align*} establishing a Starkey Rule for $\mathcal{H}(R_n)$.

Now we do the same thing for the upper right corner of the character table. Let $\lambda$ stay in $\mathcal{P}_r$, but now let $\mu\in\mathcal{SP}_n$. Shoji \cite[Theorem~7.10]{shoji} has a generalization of Starkey's Rule to Ariki-Koike algebras. We haven't established the notation to state Shoji's formula, so define $t_\mu^\nu(r;q)$ to be the expression so that the right side of \cite[Theorem~7.10]{shoji} can be written $\sum_{\nu\in\mathcal{SP}_n} t_\mu^\nu(r;q)\cdot \ba{\chi}_\lambda(w_\nu)$.

Shoji uses a different presentation of the Hecke algebra than ours; using our presentation, \[\ba{\chi}^*_\nu(T_{w_\mu}) = \sum_{\gamma\in\mathcal{SP}_n} t_\mu^\gamma \cdot \ba{\chi}_\nu(w_\gamma),\] where $t_\mu^\gamma = q^{1/2} \cdot t_\mu^\nu(2;q^{1/2})|_{u_1=q,u_2=-q^{-1}}$.

The key point here is that $t_\mu^\gamma$ is independent of $\nu$. Using this and Corollary \ref{bmatrixcomputation},

\begin{align*}\chi^*_\lambda(T_\mu) &= \sum_{\nu\in\mathcal{SP}_n} \left( \sum\limits_{\substack{\alpha \text{ partition s.t.} \\ \nu_1 - \alpha \text{ horiz strip} \\ \text{of size } |\lambda|}} c_{\alpha, \nu_2}^\lambda\right) \ba{\chi}^*_\nu(T_{w_\mu}) \\&= \sum_{\nu\in\mathcal{SP}_n} \left( \sum\limits_{\substack{\alpha \text{ partition s.t.} \\ \nu_1 - \alpha \text{ horiz strip} \\ \text{of size } |\lambda|}} c_{\alpha, \nu_2}^\lambda\right) \sum_{\gamma\in\mathcal{SP}_n} t_\mu^\gamma \cdot \ba{\chi}_\nu(w_\gamma) \\&= \sum_{\gamma\in\mathcal{SP}_n} t_\mu^\gamma \left(\sum_{\nu\in\mathcal{SP}_n} \left( \sum\limits_{\substack{\alpha \text{ partition s.t.} \\ \nu_1 - \alpha \text{ horiz strip} \\ \text{of size } |\lambda|}} c_{\alpha, \nu_2}^\lambda\right)  \ba{\chi}_\nu(w_\gamma)\right) \\&=  \sum_{\gamma\in\mathcal{SP}_n} t_\mu^\gamma \cdot \chi_\lambda(r_\gamma).\end{align*}

Since Shoji's formula applies to type $A$ as well as type $C$, we could in theory use it in place of Starkey's Rule as well, but Starkey's Rule is computationally simpler.

Therefore, we have a Starkey-Shoji formula for the entire character table of $\mathcal{H}(RSp_{2n})$. For $\alpha\in\mathcal{Q}_n$, let $\mathcal{P}_\alpha$ denote whichever one of $\mathcal{SP}_n, \mathcal{P}_r$ contains $\alpha$.

\begin{theorem} \label{starkeyshojimonoid}
Let $\lambda,\mu\in\mathcal{Q}_n$. Then \[\chi^*_\lambda(T_\mu) = \sum_{\gamma\in \mathcal{P}_\mu} p_\mu^\gamma \cdot \chi_\lambda(r_\gamma),\] where \[p_\mu^\nu = \begin{cases} \delta_{\mu\gamma}, & \mathcal{P}_\mu = \mathcal{P}_\gamma \\ \sum_{\alpha\vdash r'} \frac{|C_\alpha\cap S_\mu|}{|S_\mu|} \det(q-\rho_\mu(w_\alpha)), & \mathcal{P}_\lambda = \mathcal{P}_r,\; \mathcal{P}_\mu = \mathcal{P}_{r'},\; r<r' \\ t_\mu^\gamma, & \mathcal{P}_\lambda = \mathcal{P}_r,\; \mathcal{P}_\mu = \mathcal{SP}_n, \\ 0, & \text{else}. \end{cases}\]
\end{theorem}

\begin{remark}
The reader may notice that a lot of the specifics here do not matter: any formula of the form \[\ba{\chi}_\lambda^*(T_{w_\mu}) = \sum_{\gamma\in\mathcal{P}_\mu} p_\mu^\gamma\cdot \ba{\chi}_\lambda(w_\mu)\] becomes an exactly analogous formula for $\chi^*_\lambda(T_\mu)$.
\end{remark}

\subsection{Explicit Computations}

\begin{ex}
In the case $R=RSp_4$, we obtain the following. Let $M_q$ be the character table of $\mathcal{H}(RSp_4)$, and let $Y_q$ and $B$ be the associated Y- and B- matrices.

\[ Y_q =
\begin{pmatrix}
1 & 2 & 1 & 1 & 1 & 0 & 0 & 0 & 0 \\
q & q-1 & -1 & q & -1 & 0 & 0 & 0 & 0 \\
q^2 & -2q^2 & 1 & q^4 & q^2 & 0 & 0 & 0 & 0 \\
-1 & q-1 & -1 & q & q & 0 & 0 & 0 & 0 \\
-q & 0 & 1 & q^2 & -q & 0 & 0 & 0 & 0 \\
0 & 0 & 0 & 0 & 0 & 1 & 1 & 0 & 0 \\
0 & 0 & 0 & 0 & 0 & q & -1 & 0 & 0 \\
0 & 0 & 0 & 0 & 0 & 0 & 0 & 1 & 0 \\
0 & 0 & 0 & 0 & 0 & 0 & 0 & 0 & 1
\end{pmatrix}, \hspace{20pt} B =
\begin{pmatrix}
1 & 0 & 0 & 0 & 0 & 0 & 1 & 1 & 0 \\
0 & 1 & 0 & 0 & 0 & 1 & 1 & 1 & 0 \\
0 & 0 & 1 & 0 & 0 & 0 & 1 & 0 & 0 \\
0 & 0 & 0 & 1 & 0 & 1 & 0 & 1 & 1 \\
0 & 0 & 0 & 0 & 1 & 1 & 0 & 0 & 0 \\
0 & 0 & 0 & 0 & 0 & 1 & 0 & 1 & 1 \\
0 & 0 & 0 & 0 & 0 & 0 & 1 & 1 & 0 \\
0 & 0 & 0 & 0 & 0 & 0 & 0 & 1 & 1 \\
0 & 0 & 0 & 0 & 0 & 0 & 0 & 0 & 1
\end{pmatrix},
\]

\[M_q = Y_qB =
\kbordermatrix{ & (1^2,0) & (1,1) & (0,1^2) & (2,0) & (0,2) & (2) & (1^2) & (1) & (0) \\
(1^2,0) & 1 & 2 & 1 & 1 & 1 & 4 & 4 & 4 & 1 \\
(1,1) & q & q-1 & -1 & q & -1 & 2q-2 & 2q-2 & 3q-1 & q \\
(0,1^2) & q^2 & -2q^2 & 1 & q^4 & q^2 & q^4-q^2 & -q^2+1 & q^4-q^2 & q^4 \\
(2,0) & -1 & q-1 & -1 & q & q & 3q-1 & q-3 & 2q-2 & q \\
(0,2) & -q & 0 & 1 & q^2 & -q & q^2-q & -q+1 & q^2-q & q^2 \\
(2) & 0 & 0 & 0 & 0 & 0 & 1 & 1 & 2 & 1 \\
(1^2) & 0 & 0 & 0 & 0 & 0 & q & -1 & q-1 & q \\
(1) & 0 & 0 & 0 & 0 & 0 & 0 & 0 & 1 & 1 \\
(0) & 0 & 0 & 0 & 0 & 0 & 0 & 0 & 0 & 1
}. \]

Note that when we send $q\mapsto 1$ in $M_q$, we obtain the character table $M$ of $RSp_4$.

\end{ex}

\begin{ex} \label{nonstandardelementexample}
We demonstrate a computation of character values of a nonstandard element of $\mathcal{H}(RSp_{2n})$. Take the element $T_r\in\mathcal{H}(RSp_4)$, where \[r = \begin{bmatrix} 0&0&0&0 \\ 0&0&0&0 \\ 0&1&0&0 \\ 0&0&0&0\end{bmatrix}.\] Notice that $r$ is nilpotent but nonzero, so $T_r$ is not a standard element. We can write $S = \{s,t\}$,\; $\Lambda = \{0,e,f,I\}$ where \[s = \begin{bmatrix} &1\\1\\&&&1\\&&1\end{bmatrix}, \hspace{20pt} t = \begin{bmatrix}&&&1\\&1\\&&1\\1\end{bmatrix}, \hspace{20pt} e = \begin{bmatrix} 1\\&0\\&&0\\&&&0\end{bmatrix}, \hspace{20pt} f = \begin{bmatrix} 1\\&1\\&&0\\&&&0\end{bmatrix}.\] Then $r = stes$ is a minimal expression for $r$. Therefore, the computations in Section \ref{nonstandardelements} yield \[T_r = T_sT_tT_eT_s \equiv T_e(T_s)^2 T_t = T_e((q-1)T_s+qT_1)T_t = (q-1)T_{est} + qT_{et}.\] Now, \[T_{est} \equiv q^{l(est)-l(este)}T_{este} = qT_0 \hspace{20pt} \text{and} \hspace{20pt} T_{et} \equiv q^{l(et)-l(ete)}T_{ete} = qT_0,\] so \[T_r\equiv (2q^2-q)T_0.\] Therefore, $\chi_\emptyset(T_r) = 2q^2-q$, and $\chi_\lambda(T_r) = 0$ for all other $\lambda\in\mathcal{Q}_n$.
\end{ex}

\begin{ex}
The following is the character table of $\mathcal{H}(R_{2})$.
\begin{equation*}
M_q = \begin{pmatrix}
1 & 1 & 2 & 1  \\
q & -1 & q-1 & q  \\
0 & 0 & 1 & 1  \\
0 & 0 & 0 & 1  
\end{pmatrix}.
\end{equation*}

Under the decomposition $M_q=A_qY_q$, we get
\begin{equation*}
A_q = \begin{pmatrix}
1 & 0 & 2 & 1  \\
0 & 1 & q-1 & q  \\
0 & 0 & 1 & 1  \\
0 & 0 & 0 & 1  
\end{pmatrix},
\end{equation*}
so $A_q$ depends on $q$. It is not clear how to compute the A-matrix directly, which is the reason our Hecke algebra computations used the B-matrix instead of the A-matrix.
\end{ex}

\section{Further Questions} \label{furtherquestionssection}
To our knowledge, there is no combinatorial formula for the A- or B-matrix of a Renner monoid of types other than $A$ and $C$ (in particular, in type $D$). This seems like a tractable problem, and would allow nice formulas for the Hecke algebra character tables of other types as well.

Godelle's definition of the generic Hecke algebra is a generalization of the single-parameter generic Hecke algebra for a finite Weyl group. However, types $C$, $F$, and $G$ have two-parameter generic Weyl group Hecke algebras. A generalization of Godelle's definition to the two-parameter case would perhaps allow us to port our computations wholesale, and compute the character table of those Hecke algebras.

Another possible generalization would be to the monoid equivalent of Ariki-Koike algebras, which are the Hecke algebras of the complex reflection groups $(\Z/k\Z)\wr S_n$. The set of $n\times n$ monomial matrices with at most one entry in each row and column, and nonzero entries equaling $k$-th roots of unity form an inverse monoid with unit group $(\Z/k\Z)\wr S_n$, so it may be possible to define an associated Hecke algebra, and to extend our Starkey-Shoji formula to these algebras.

In addition, we exploited the conception of the B-matrix in terms of multiplicities to calculate Hecke algebra character tables of Renner monoids. The picture is not so simple using the A-matrix: the entries in the Hecke algebra case depend on $q$. However, the A-matrix is combinatorially interesting for Renner monoids of (at least) types $A$ and $C$; perhaps the Hecke algebra A-matrix will turn out to be a nice $q$-analogue. The interested reader may also note that the computation of the Hecke algebra A-matrix is implicit in our computations of the character table: just multiply by $Y_q^{-1}$ on the right. A nice formula for the A-matrix of $\mathcal{H}(RSp_{2n})$ could potentially lead to a Murnaghan-Nakayama formula (which has been done for $\mathcal{H}(R_n)$ by Dieng, Halverson, and Poladian in Theorem 3.4 of \cite{Dieng2003}).

\newpage
\section{Appendix}

\subsection{Character Table for $RSp_{6}$}

Let $R = RSp_6$. We compute the Renner monoid character table. Using our formulas, (Theorem \ref{amatrixcomputation}, Corollary \ref{bmatrixcomputation}), we can calculate the A- and B-matrices:

$$A = \kbordermatrix{ r_\mu \backslash \chi_\lambda & 1^3,0 & 1^2,1 & 1,1^2 & 0,1^3 & 21,0 & 1,2 & 2,1 & 0,21 & 3,0 & 0,3 & (1^3) & (21) & (3) & (1^2) & (2) & (1) & (0) \\
1^3,\emptyset & 1 & 0 & 0 & 0 & 0 & 0 & 0 & 0 & 0 & 0 & 8 & 0 & 0 & 12 & 0 & 6 & 1 \\
1^2,1 & 0 & 1 & 0 & 0 & 0 & 0 & 0 & 0 & 0 & 0 & 0 & 0 & 0 & 4 & 0 & 4 & 1 \\
1,1^2 & 0 & 0 & 1 & 0 & 0 & 0 & 0 & 0 & 0 & 0 & 0 & 0 & 0 & 0 & 0 & 2 & 1 \\
\emptyset,1^3 & 0 & 0 & 0 & 1 & 0 & 0 & 0 & 0 & 0 & 0 & 0 & 0 & 0 & 0 & 0 & 0 & 1 \\
21,\emptyset & 0 & 0 & 0 & 0 & 1 & 0 & 0 & 0 & 0 & 0 & 0 & 4 & 0 & 0 & 2 & 2 & 1 \\
1,2 & 0 & 0 & 0 & 0 & 0 & 1 & 0 & 0 & 0 & 0 & 0 & 0 & 0 & 0 & 0 & 2 & 1 \\
2,1 & 0 & 0 & 0 & 0 & 0 & 0 & 1 & 0 & 0 & 0 & 0 & 0 & 0 & 0 & 2 & 0 & 1 \\
\emptyset,21 & 0 & 0 & 0 & 0 & 0 & 0 & 0 & 1 & 0 & 0 & 0 & 0 & 0 & 0 & 0 & 0 & 1 \\ 3,\emptyset & 0 & 0 & 0 & 0 & 0 & 0 & 0 & 0 & 1 & 0 & 0 & 0 & 2 & 0 & 0 & 0 & 1 \\
\emptyset,3 & 0 & 0 & 0 & 0 & 0 & 0 & 0 & 0 & 0 & 1 & 0 & 0 & 0 & 0 & 0 & 0 & 1 \\
(1^3) & 0 & 0 & 0 & 0 & 0 & 0 & 0 & 0 & 0 & 0 & 1 & 0 & 0 & 3 & 0 & 3 & 1 \\
(21) & 0 & 0 & 0 & 0 & 0 & 0 & 0 & 0 & 0 & 0 & 0 & 1 & 0 & 0 & 1 & 1 & 1 \\
(3) & 0 & 0 & 0 & 0 & 0 & 0 & 0 & 0 & 0 & 0 & 0 & 0 & 1 & 0 & 0 & 0 & 1 \\
(1^2) & 0 & 0 & 0 & 0 & 0 & 0 & 0 & 0 & 0 & 0 & 0 & 0 & 0 & 1 & 0 & 2 & 1 \\
(2) & 0 & 0 & 0 & 0 & 0 & 0 & 0 & 0 & 0 & 0 & 0 & 0 & 0 & 0 & 1 & 0 & 1 \\
(1) & 0 & 0 & 0 & 0 & 0 & 0 & 0 & 0 & 0 & 0 & 0 & 0 & 0 & 0 & 0 & 1 & 1 \\
(0) & 0 & 0 & 0 & 0 & 0 & 0 & 0 & 0 & 0 & 0 & 0 & 0 & 0 & 0 & 0 & 0 & 1
},$$
$$B = \kbordermatrix{ r_\mu \backslash \chi_\lambda & 1^3,0 & 1^2,1 & 1,1^2 & 0,1^3 & 21,0 & 1,2 & 2,1 & 0,21 & 3,0 & 0,3 & (1^3) & (21) & (3) & (1^2) & (2) & (1) & (0) \\
1^3,\emptyset & 1 & 0 & 0 & 0 & 0 & 0 & 0 & 0 & 0 & 0 & 1 & 0 & 0 & 1 & 0 & 0 & 0 \\
1^2,1 & 0 & 1 & 0 & 0 & 0 & 0 & 0 & 0 & 0 & 0 & 1 & 1 & 0 & 1 & 1 & 0 & 0 \\
1,1^2 & 0 & 0 & 1 & 0 & 0 & 0 & 0 & 0 & 0 & 0 & 1 & 1 & 0 & 1 & 0 & 0 & 0 \\
\emptyset,1^3 & 0 & 0 & 0 & 1 & 0 & 0 & 0 & 0 & 0 & 0 & 1 & 0 & 0 & 0 & 0 & 0 & 0 \\
21,\emptyset & 0 & 0 & 0 & 0 & 1 & 0 & 0 & 0 & 0 & 0 & 0 & 1 & 0 & 1 & 1 & 1 & 0 \\
1,2 & 0 & 0 & 0 & 0 & 0 & 1 & 0 & 0 & 0 & 0 & 0 & 1 & 1 & 0 & 1 & 0 & 0 \\
2,1 & 0 & 0 & 0 & 0 & 0 & 0 & 1 & 0 & 0 & 0 & 0 & 1 & 1 & 1 & 1 & 1 & 0 \\
\emptyset,21 & 0 & 0 & 0 & 0 & 0 & 0 & 0 & 1 & 0 & 0 & 0 & 1 & 0 & 0 & 0 & 0 & 0 \\ 3,\emptyset & 0 & 0 & 0 & 0 & 0 & 0 & 0 & 0 & 1 & 0 & 0 & 0 & 1 & 0 & 1 & 1 & 1 \\
\emptyset,3 & 0 & 0 & 0 & 0 & 0 & 0 & 0 & 0 & 0 & 1 & 0 & 0 & 1 & 0 & 0 & 0 & 0 \\
(1^3) & 0 & 0 & 0 & 0 & 0 & 0 & 0 & 0 & 0 & 0 & 1 & 0 & 0 & 1 & 0 & 0 & 0 \\
(21) & 0 & 0 & 0 & 0 & 0 & 0 & 0 & 0 & 0 & 0 & 0 & 1 & 0 & 1 & 1 & 1 & 0 \\
(3) & 0 & 0 & 0 & 0 & 0 & 0 & 0 & 0 & 0 & 0 & 0 & 0 & 1 & 0 & 1 & 1 & 1 \\
(1^2) & 0 & 0 & 0 & 0 & 0 & 0 & 0 & 0 & 0 & 0 & 0 & 0 & 0 & 1 & 0 & 1 & 0 \\
(2) & 0 & 0 & 0 & 0 & 0 & 0 & 0 & 0 & 0 & 0 & 0 & 0 & 0 & 0 & 1 & 1 & 1 \\
(1) & 0 & 0 & 0 & 0 & 0 & 0 & 0 & 0 & 0 & 0 & 0 & 0 & 0 & 0 & 0 & 1 & 1 \\
(0) & 0 & 0 & 0 & 0 & 0 & 0 & 0 & 0 & 0 & 0 & 0 & 0 & 0 & 0 & 0 & 0 & 1
}.$$

One can also compute:

$$Y = \kbordermatrix{ r_\mu \backslash \chi_\lambda & 1^3,0 & 1^2,1 & 1,1^2 & 0,1^3 & 21,0 & 1,2 & 2,1 & 0,21 & 3,0 & 0,3 & (1^3) & (21) & (3) & (1^2) & (2) & (1) & (0) \\
1^3,\emptyset & 1 & 3 & 3 & 1 & 2 & 3 & 3 & 2 & 1 & 2 & 0 & 0 & 0 & 0 & 0 & 0 & 0 \\
1^2,1 & 1 & 1 & -1 & -1 & 2 & -1 & 1 & -2 & 1 & -1 & 0 & 0 & 0 & 0 & 0 & 0 & 0 \\
1,1^2 & 1 & -1 & -1 & 1 & 2 & -1 & -1 & 2 & 1 & 1 & 0 & 0 & 0 & 0 & 0 & 0 & 0 \\
\emptyset,1^3 & 1 & -3 & 3 & -1 & 2 & 3 & -3 & -2 & 1 & -1 & 0 & 0 & 0 & 0 & 0 & 0 & 0 \\
21,\emptyset & -1 & -1 & -1 & -1 & 0 & 1 & 1 & 0 & 1 & 1 & 0 & 0 & 0 & 0 & 0 & 0 & 0 \\
1,2 & -1 & -1 & 1 & 1 & 0 & -1 & 1 & 0 & 1 & -1 & 0 & 0 & 0 & 0 & 0 & 0 & 0 \\
2,1 & -1 & 1 & -1 & 1 & 0 & 1 & -1 & 0 & 1 & -1 & 0 & 0 & 0 & 0 & 0 & 0 & 0 \\
\emptyset,21 & -1 & 1 & 1 & -1 & 0 & -1 & -1 & 0 & 1 & 1 & 0 & 0 & 0 & 0 & 0 & 0 & 0 \\ 3,\emptyset & 1 & 0 & 0 & 1 & -1 & 0 & 0 & -1 & 1 & 1 & 0 & 0 & 0 & 0 & 0 & 0 & 0 \\
\emptyset,3 & 1 & 0 & 0 & -1 & -1 & 0 & 0 & 1 & 1 & -1 & 0 & 0 & 0 & 0 & 0 & 0 & 0 \\
(1^3) & 0 & 0 & 0 & 0 & 0 & 0 & 0 & 0 & 0 & 0 & 1 & 2 & 1 & 0 & 0 & 0 & 0 \\
(21) & 0 & 0 & 0 & 0 & 0 & 0 & 0 & 0 & 0 & 0 & -1 & 0 & 1 & 0 & 0 & 0 & 0 \\
(3) & 0 & 0 & 0 & 0 & 0 & 0 & 0 & 0 & 0 & 0 & 1 & -1 & 1 & 0 & 0 & 0 & 0 \\
(1^2) & 0 & 0 & 0 & 0 & 0 & 0 & 0 & 0 & 0 & 0 & 0 & 0 & 0 & 1 & 1 & 0 & 0 \\
(2) & 0 & 0 & 0 & 0 & 0 & 0 & 0 & 0 & 0 & 0 & 0 & 0 & 0 & -1 & 1 & 0 & 0 \\
(1) & 0 & 0 & 0 & 0 & 0 & 0 & 0 & 0 & 0 & 0 & 0 & 0 & 0 & 0 & 0 & 1 & 0 \\
(0) & 0 & 0 & 0 & 0 & 0 & 0 & 0 & 0 & 0 & 0 & 0 & 0 & 0 & 0 & 0 & 0 & 1
}.$$

Therefore, we have:

\begin{align*} &M(RSp_6) = AY = YB =\\&
\kbordermatrix{ r_\mu \backslash \chi_\lambda & 1^3,0 & 1^2,1 & 1,1^2 & 0,1^3 & 21,0 & 1,2 & 2,1 & 0,21 & 3,0 & 0,3 & (1^3) & (21) & (3) & (1^2) & (2) & (1) & (0) \\
1^3,\emptyset & 1 & 3 & 3 & 1 & 2 & 3 & 3 & 2 & 1 & 2 & 8 & 16 & 8 & 12 & 12 & 6 & 1 \\
1^2,1 & 1 & 1 & -1 & -1 & 2 & -1 & 1 & -2 & 1 & -1 & 0 & 0 & 0 & 4 & 4 & 4 & 1 \\
1,1^2 & 1 & -1 & -1 & 1 & 2 & -1 & -1 & 2 & 1 & 1 & 0 & 0 & 0 & 0 & 0 & 2 & 1 \\
\emptyset,1^3 & 1 & -3 & 3 & -1 & 2 & 3 & -3 & -2 & 1 & -1 & 0 & 0 & 0 & 0 & 0 & 0 & 1 \\
21,\emptyset & -1 & -1 & -1 & -1 & 0 & 1 & 1 & 0 & 1 & 1 & -4 & 0 & 4 & -2 & 2 & 2 & 1 \\
1,2 & -1 & -1 & 1 & 1 & 0 & -1 & 1 & 0 & 1 & -1 & 0 & 0 & 0 & 0 & 0 & 2 & 1 \\
2,1 & -1 & 1 & -1 & 1 & 0 & 1 & -1 & 0 & 1 & -1 & 0 & 0 & 0 & -2 & 2 & 0 & 1 \\
\emptyset,21 & -1 & 1 & 1 & -1 & 0 & -1 & -1 & 0 & 1 & 1 & 0 & 0 & 0 & 0 & 0 & 0 & 1 \\ 3,\emptyset & 1 & 0 & 0 & 1 & -1 & 0 & 0 & -1 & 1 & 1 & 2 & -2 & 2 & 0 & 0 & 0 & 1 \\
\emptyset,3 & 1 & 0 & 0 & -1 & -1 & 0 & 0 & 1 & 1 & -1 & 0 & 0 & 0 & 0 & 0 & 0 & 1 \\
(1^3) & 0 & 0 & 0 & 0 & 0 & 0 & 0 & 0 & 0 & 0 & 1 & 2 & 1 & 3 & 3 & 3 & 1 \\
(21) & 0 & 0 & 0 & 0 & 0 & 0 & 0 & 0 & 0 & 0 & -1 & 0 & 1 & -1 & 1 & 1 & 1 \\
(3) & 0 & 0 & 0 & 0 & 0 & 0 & 0 & 0 & 0 & 0 & 1 & -1 & 1 & 0 & 0 & 0 & 1 \\
(1^2) & 0 & 0 & 0 & 0 & 0 & 0 & 0 & 0 & 0 & 0 & 0 & 0 & 0 & 1 & 1 & 2 & 1 \\
(2) & 0 & 0 & 0 & 0 & 0 & 0 & 0 & 0 & 0 & 0 & 0 & 0 & 0 & -1 & 1 & 0 & 1 \\
(1) & 0 & 0 & 0 & 0 & 0 & 0 & 0 & 0 & 0 & 0 & 0 & 0 & 0 & 0 & 0 & 1 & 1 \\
(0) & 0 & 0 & 0 & 0 & 0 & 0 & 0 & 0 & 0 & 0 & 0 & 0 & 0 & 0 & 0 & 0 & 1
},\end{align*}

which matches the character table as calculated via the method of Li, Li, and Cao \cite{li2008representations}.

\newpage
\nocite{*}
%
\begingroup 
\let\itshape\upshape
\bibliographystyle{ieeetr-noquotes2}
\bibliography{Bibliography.bib}
\endgroup
\Addresses

\end{document}